\documentclass[a4paper]{amsart}
\usepackage{amsthm}
\usepackage{amsmath}
\usepackage{amssymb}
\usepackage{latexsym}
\usepackage{enumerate}
\usepackage{graphicx}
\usepackage[utf8]{inputenc}

\usepackage{epsfig} 
\usepackage{pgf}
\usepackage{tikz}
\usepackage{float}
\usepackage{dsfont}
\usepackage{times}

\newtheorem{theoremA}{Theorem}

\renewcommand{\thetheoremName}

\newtheorem*{thm}{Theorem}

\newtheorem{proposition[[]]}[theoremName]{Proposition G}

\newtheorem{theorem}{Theorem}[section]
\newtheorem{lemma}[theorem]{Lemma}

\newtheorem{proposition}[theorem]{Proposition}
\newtheorem{corollary}[theorem]{Corollary}

\newtheorem{corollaryA}[theoremA]{Corollary}

\theoremstyle{definition}
\newtheorem{definition}[theorem]{Definition}
\newtheorem{example}[theorem]{Example}

\newtheorem{remark}{Remark}

\numberwithin{equation}{section}

\newcommand{\eq}{\begin{equation}}
\newcommand{\eeq}{\end{equation}}
\newcommand{\al}{\begin{aligned}}
\newcommand{\aal}{\end{aligned}}
\newcommand{\RT}{\mathbb{R}^3}
\newcommand{\HT}{\mathbb{H}^3}
\newcommand{\ST}{\mathbb{S}^3}
\newcommand{\MT}{\mathbb{M}^3(\kappa)}

\begin{document}

\title[Conformal type of ends]{Conformal type of ends of revolution in space forms of constant sectional curvature.}
\author[V. Gimeno]{Vicent Gimeno}      
\address{Department of Mathematics-IMAC, Universitat Jaume I, Castell\'o de la Plana, Spain                        
}
\email{gimenov@uji.es}
\thanks{Work partially supported by the Research Program of University Jaume I Project P1-1B2012-18, and DGI -MINECO grant (FEDER) MTM2013-48371-C2-2-P}
\author[I. Gozalbo]{Irmina Gozalbo}
\email{igozalbo@gmail.com}

\begin{abstract}
In this paper we consider the conformal type (parabolicity or non-parabolicity) of complete ends of revolution immersed in simply connected space forms of constant sectional curvature. We show that any complete end of revolution in the $3$-dimensional Euclidean  space or in the $3$-dimensional sphere is parabolic. In the case of ends of revolution in the hyperbolic $3$-dimensional space, we find sufficient conditions to attain parabolicity for complete ends of revolution using their relative position to the complete flat surfaces of revolution.
\end{abstract}
\maketitle
\setcounter{tocdepth}{1}
\tableofcontents
\section{Introduction.}

Let $\Sigma$ be a complete and non-compact surface. Let $D\subset \Sigma$ be an open precompact subset of $\Sigma$ with smooth boundary. An end $E$ of $\Sigma$ with respect to $D$ is a connected unbounded component of $\Sigma\setminus D$. An end $E$ is \emph{parabolic} (\cite{Meeks-Perez-2,Li2000, GriExp}) if every bounded harmonic function on $E$ is determined by its boundary values. 

This paper is concerned with the study of the conformal type (parabolicity or non-parabolicity) of complete ends of revolution immersed in the $3$-dimensional Euclidean space $\RT$, in the $3$-dimensional hyperbolic space $\HT$, or in the $3$-dimensional sphere $\ST$. Let us denote by $\MT$ the simply connected space form of constant sectional curvature $\kappa\in \mathbb{R}$. Hence, $\mathbb{M}^3(1)=\ST$, $\mathbb{M}^3(0)=\RT$, $\mathbb{M}^3(-1)=\HT$. An end of a complete surface in $\MT$ is a \emph{complete end of revolution} if  there exists a geodesic in $\MT$ such that the end is invariant by the group of rotations of $\MT$ that leaves this geodesic point-wise fixed. More precisely, an end of revolution will be the rotation along a geodesic ray $\gamma$ of $\MT$ of a generating smooth curve $\beta:[0,\infty)\to\mathbb{M}^2(\kappa)$ contained in a totally geodesic hypersurface  $\mathbb{M}^2(\kappa)$  where the ray $\gamma$ belongs. In order to guarantee the smoothness and that the end is the end of a complete surface  we require that the generating curve $\beta$ be regular, with infinite longitude and  does not intersect the geodesic ray $\gamma$. 

The conformal type of a Riemannian manifold has been largely studied. In particular in \cite{Troya99} sufficient and necessary conditions for the parabolicity of a manifold with a warped cylindrical end were provided, in \cite{GriExp} rotationally symmetric manifolds were analyzed, and in the examples of \cite{HP,MP3} certain surfaces of revolution in $\RT$ have been studied from an extrinsic approach. Our first result characterizes the conformal type of complete ends of revolution in $\RT$ or in $\ST$

\begin{theoremA}\label{teounua}
Any complete end of revolution in $\RT$, or in $\ST$, is a parabolic end.
\end{theoremA}

The conformal classification of  ends of revolution in $\HT$ becomes more complicated. In the hyperbolic space there is no restriction on the conformal type of ends of revolution. Actually, in Section \ref{ex} we will show examples of parabolic and non-parabolic ends in $\HT$. In the half space model of the hyperbolic space,
\begin{equation}
\HT: = \left\{ (x_1,x_2,x_3) \in \RT\,:\,x_3>0 \right\},\quad g_{\HT}= \frac{1}{x_3^2} \left(dx_1^2 +  dx_2^2+ dx_3^2  \right),
\end{equation}
we can provide a sufficient condition for parabolicity using a $c$-{\em cone}. Given $c\in\mathbb{R}_+$, a $c$-cone is the rotation along the $z_3$ axis of curve
\begin{equation}
\beta:(0,\infty)\to \HT,\quad \beta(t)=(t,0,ct).
\end{equation}
\begin{figure}
   \centering
   \includegraphics[scale=0.25]{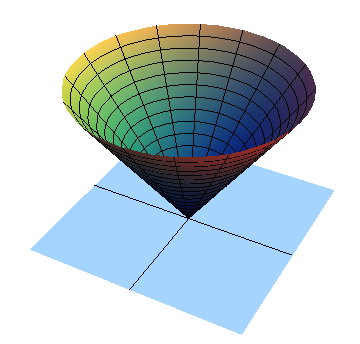}\quad\includegraphics[scale=0.25]{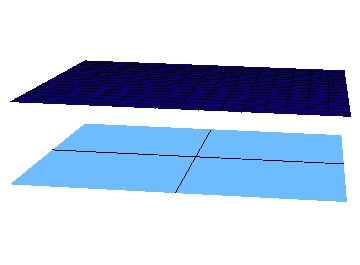}\quad\includegraphics[scale=0.25]{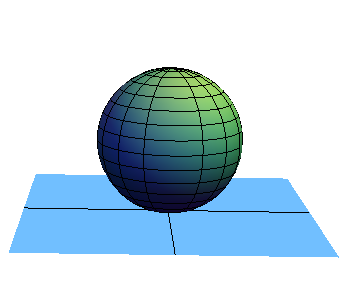}
   \caption{The $c$-cone (or the set of points at a fixed distance from a geodesic) and the horosphere are the only complete flat surfaces immersed in $\HT$. Everyone of these flat surfaces divides $\HT$ in two parts. If an end of revolution is contained on a $c$-cone, or on (inside) a horosphere (horoball), the end is parabolic.}
   \label{cone}
\end{figure}Any $c$-cone divides the hyperbolic space $\HT$ in two parts (see figure \ref{cone}). One part on the $c$-cone and the other part down the $c$-cone. Using this partition property of the $c$-cones we can state the following Theorem

\begin{theoremA}\label{teoendh3}
Let $E$ be a complete end of revolution in $\HT$. Suppose that the end $E$ is contained on  a $c$-cone for some $c>0$. Then, the end $E$ is parabolic.
\end{theoremA} 

An other sufficient condition can be provided using horospheres

\begin{theoremA}\label{slice}
Let $E$ be a complete end of revolution in $\HT$. Suppose that the end $E$ is contained  on  the horosphere $\{x_3=z\}$  for some $z>0$. Then, the end $E$ is parabolic.
\end{theoremA} 
By using the above Theorem we can characterize the conformal type of ends of revolution immersed inside of a compact set of the hyperbolic space.
\begin{corollaryA}\label{corendH3}
Let $E$ be a complete end of revolution in $\HT$ contained in a compact set of $\HT$. Then, $E$ is a parabolic end.
\end{corollaryA}
Moreover, theorem  \ref{slice} allow us to know that complete non-parabolic ends of revolution in $\HT$ approaches to the $\{x_3=0\}$ plane.
\begin{corollaryA}\label{corendhyp}
Let $E$ be a complete and non-parabolic end of revolution in $\HT$. Then, 
$$
\inf_{p\in E}x_3(p)=0.
$$
\end{corollaryA}

The conformal type of a surface is related with the transcience or recurrence of the Brownian motion. The simplest way to construct Brownian motion on a surface is to construct first the heat kernel which will serve as the density of the transition probability.

Given a surface $\Sigma$ with Laplacian $\triangle$, the heat kernel on $\Sigma$ is a function $p(t,x,y)$ on $(0,\infty)\times \Sigma\times \Sigma$ which is the minimal positive fundamental solution to the heat equation
\begin{equation}
\frac{\partial v}{\partial t}=\triangle v.
\end{equation}
In other words, the Cauchy problem
\begin{equation}
\left\{
\begin{aligned}
\frac{\partial v}{\partial t}&=\triangle v\\
v\vert_{t=0}&=v_0(x)
\end{aligned}
\right.
\end{equation}
has solution
\begin{equation}
v(x,t)=\int_{\Sigma}p(t,x,y)v_0(y)dA(y),
\end{equation}
where $dA$ is the volume element of $\Sigma$. The Brownian motion on a surface is called recurrent if it visits any open set at arbitrarily large moments of time with probability $1$, and transcient otherwise. The Brownian motion on $\Sigma$ is transcient (see \cite{GriExp,Ahl,Khas}) if $$
\int_1^{\infty} p(t,x,x)dt < \infty
$$ Otherwise, the Brownian motion on  $\Sigma$ is recurrent. 
Given a complete and non-compact surface $\Sigma$ and an open precompact set $D\subset \Sigma$, the Brownian motion is recurrent, if and only if, every end of $\Sigma$ with respect to $D$ is parabolic. 
 
Another property of the Brownian motion to be considered in this paper is stochastic completeness. This is a property of a stochastic process to have infinite lifetime. In other words, a process is stochastically complete if the total probability of the particle being found in the state space is constantly equal to $1$. For the Brownian motion this means
\begin{equation}
\int_\Sigma p(t,x,y)dA(y)=1,
\end{equation}
for any $t>0$. Namely, the heat kernel is an authentic probability measure. For complete ends of revolution we can state 

\begin{theoremA}\label{teoE} Let $\Sigma$ be complete and non-compact surface of finite topological type immersed in $\MT$ with $\kappa\in\mathbb{R}$. Suppose that there exists a compact subset $\Omega\subset \Sigma$ of $\Sigma$ such that every end of $\Sigma$ with respect to $\Omega$ is an end of revolution  in $\MT$. Then, $\Sigma$  is stochastically complete. \end{theoremA}

Hoffman and Meeks proved in \cite{MeHo1990} that a properly immersed minimal surface in $\mathbb{R}^3$ disjoint from a plane is a plane. Otherwise stated, if $M$ is a minimal surface properly immersed in $\mathbb{R}^3$ and for some $c>0$, $M\cap\{z>c\}\neq \emptyset$, then either $M\cap\{z=c\}\neq \emptyset$, or $M$ is a plane parallel to $\{z=c\}$. 

For the hyperbolic space, Rodriguez and Rosenberg proved in \cite{Rod98} that every  constant mean curvature one surface $M$, properly embedded in a horoball  $B\subset\HT$ such that, $M\cap \partial B=\emptyset$, is a horosphere.    As a surprising corollary of Theorem \ref{teoE} we obtain
\begin{theoremA}\label{horoteo}Let $M$ be a complete non-compact surface of revolution properly immersed in $\HT$. Suppose $M\cap B\neq \emptyset$ for some horoball $B\subset \HT$ then,
\begin{enumerate}
\item if $M$ has negative sectional curvature, $M\cap \partial B\neq \emptyset$ (otherwise stated, $M$ touches the horosphere $\partial B$).
\item If $M$ has constant non-positive sectional curvature and $M\cap \partial B=\emptyset$, $M$ is a horosphere.
\item If $M$ has constant mean curvature with $\Vert \vec H\Vert\leq 1$, and $M\cap \partial B=\emptyset$, $M$ is a horosphere.
\end{enumerate}
\end{theoremA}

\subsection{outline of the paper}
The structure of the paper is as follows

In Section \ref{pre} we introduce the definitions of complete end of revolution and we study the relation with the isometry and isotropy group of $\MT$. In Theorem \ref{theoprelim} and corollary \ref{cor2.9} we prove that any complete end of revolution can be considered as a submanifold smoothly immersed in $\MT$, and intrinsically each end of revolution is endowed with a warped product metric. Indeed, in Corollary \ref{cor2.11} is proved that any end of revolution is isometric to a rotationally symmetric $2$-dimensional manifold where a geodesic ball is subtracted. That allow us, by using the well known criteria for parabolicity of rotationally symmetric model manifolds, to obtain Theorem \ref{suficient} and Corollary \ref{cor2.14} where sufficient and necessary conditions for the parabolicity in terms of the warped function of each end of revolution are provided. In Subsection \ref{confmodel} making use of conformal models of $\MT$ we obtain the explicit expressions of such warping functions. With these techniques we can prove Theorems \ref{teounua}, \ref{teoendh3}, \ref{slice}, \ref{teoE} and \ref{horoteo} in Sections \ref{proofA}, \ref{proofB}, \ref{proofC}, \ref{proofE}, and \ref{proofF} respectively. Finally, Section \ref{ex} deals with several examples of application of the main Theorems.

\section{Preliminaries.} \label{pre}
\subsection{Isotropy group and Ends of revolution in $\MT$}The only (up to scaling) $3$-dimensional simply connected Riemannian manifolds with a $6$-dimensional isometry group are:
\begin{enumerate}
\item The Euclidean space $\RT$ with vanishing curvature.
\item The hyperbolic space $\HT$ with constant sectional curvature on each tangent plane $\kappa_{\HT}=-1$.
\item The sphere $\ST$ with constant sectional curvature on each tangent plane $\kappa_{\ST}=1$.
\end{enumerate}
In this paper we denote by $\MT$, the simply connected space form of constant sectional curvature $\kappa\in \mathbb{R}$. On each point $p\in \MT$ the isotropy subgroup ${\rm stab}(p)$ of the isometry group at $p$ is $O(3)$, actually, see \cite{Petersen} for instance, $\RT=\left(\mathbb{R}^3\rtimes O(3)\right)/O(3)$, $\ST=O(4)/O(3)$, $\HT=O(1,3)/O(3)$. 

Given a point $p\in \MT$, and an unit vector $v\in T_p\MT$. We will denote by $\gamma_v$ the geodesic curve starting at $p$ with direction $v$, namely
\begin{equation}
\gamma_v:\mathbb{R}\to \MT,\quad t\to {\rm exp}_p(tv).
\end{equation} 
Since any isometry $\varphi \in {\rm stab}(p)$ fixes $p$, the pushfordward $\varphi_*:T_p\MT\to T_p\MT$ induces an automorphism in $T_p\MT$. That allow us to obtain the faithful linear isotropy representation $\rho:{\rm stab(p)}\to {\rm GL}(T_p\MT)$. Let us define
\begin{equation}
\mathcal{R}_v:=\left\{\varphi \in {\rm stab}(p)\, :\, \varphi_*(v)=v,\, {\rm det}(\varphi_*)=1\right\}. 
\end{equation}
If we choose the orthonormal basis $\{v, E_1, E_2\}$ of $T_p\MT$, for any $\varphi\in \mathcal{R}_v$ there exists $\theta\in [0,2\pi)$ such that
\begin{equation}\label{matrix}
\rho(\varphi)=\begin{pmatrix}
1&0&0\\
0&\cos(\theta)&-\sin(\theta)\\
0&\sin(\theta)&\cos(\theta)
\end{pmatrix}
\end{equation}
Then $\mathcal{R}_v$ is a Lie group isomorphic (and diffeormorphic) to $SO(2)$ and can be understood as a rotation along $\gamma_v$ because for any $\varphi\in\mathcal{R}_v$ we have $\varphi\circ\gamma_v(t)=\varphi(\exp(vt))=\exp(\varphi_*(v)t)=\exp(v t)=\gamma_v(t)$. We can show not only that the points of $\gamma_v$ are left fixed by the action of $\mathcal{R}_v$, but actually that $\gamma_v$ is the set of fixed points of $\mathcal{R}_v$. In other words,
\begin{proposition}
Let $p$ be a point of $\MT$, let $v$ be a vector in $T_p\MT$. Then, $\mathcal{R}_v$ acts freely on $\MT\setminus (\gamma_v(\mathbb{R}))$.
\end{proposition}
\begin{proof}
We are going to prove that the set of fixed points for $\mathcal{R}_v$ is precisely $\gamma_v(\mathbb{R})$.  Observe that, by using Theorem 5.1 of \cite{KobaTrans}, each connected component of the set of fixed points of $\mathcal{R}_v$ is a closed totally geodesic submanifold of $\MT$. Moreover, we can deduce that the set of fixed points has only one connected component $C_0\ni p$. Because the connected component $C_0$ of the set of fixed points which contains $p$  contains $\gamma_v(\mathbb{R})$ as well, and it therefore contains the cut points of $p$ (if $\kappa>0$). But by  corollary 5.2 of \cite{KobaTrans} any other connected component besides $C_0$ of the set of fixed points should be formed by cut points of $p$ (which belong to $C_0$). Hence, we conclude that the set of fixed points of  $\mathcal{R}_v$ has only one connected totally geodesic submanifold which contains $p$ and $\gamma_v(\mathbb{R})$. 

We can prove now that $C_0=\gamma_v(\mathbb{R})$. Because otherwise since $C_0$ is a totally geodesic submanifold, we should have an other vector $v_2\in T_p\MT$ non proportional to $v$, such that for the geodesic $\gamma_{v_2}(t)=\exp_p(v_2t)$ 
\begin{equation}
\varphi\circ \gamma_{v_2}(\mathbb{R})=\gamma_{v_2}(\mathbb{R}),\quad \forall \varphi\in \mathcal{R}_v.
\end{equation} 
 But, since the geodesic curve $\varphi\circ\gamma_{v_2}$ at $p$ is tangent to $\varphi_*(v_2)$, that means that
\begin{equation}
\varphi_*({v_2})=v_2,\quad \forall \varphi\in \mathcal{R}_v.
\end{equation} 
Set $v_2=v_2^1v+v_2^2E_1+v_2^2E_3$ the decomposition of $v_2$ in the basis $\{v,E_1,E_2\}$, hence by using the representation given in (\ref{matrix}),
\begin{equation}
\begin{pmatrix}
1&0&0\\
0&\cos(\theta)&-\sin(\theta)\\
0&\sin(\theta)&\cos(\theta)
\end{pmatrix}
\begin{pmatrix}
v_2^1\\
v_2^2\\
v_2^3
\end{pmatrix}
=\begin{pmatrix}
v_2^1\\
v_2^2\\
v_2^3
\end{pmatrix},\quad \forall \theta\in[0,2\pi).
\end{equation} 
Namely, $v_2^2=v_2^3=0$, and hence, $v_2=v_1$, but that is a contradiction because we have assumed ${\rm dim}\langle\{v,v_2\}\rangle=2$.
\end{proof}

The above proposition implies that the canonical projection
$
\pi :\MT\setminus \gamma_v(\mathbb{R})\to  (\MT\setminus \gamma_v(\mathbb{R}))/\mathcal{R}_v
$
induces a $\mathcal{R}_v$-principal fiber bundle (see \cite{Koba1} for instance). Hence,  for any $q\in \MT\setminus \gamma_v(\mathbb{R})$ the orbit space 
\begin{equation}
O_q:=\left\{\varphi(q)\, :\,\varphi\in\mathcal{R}_v\right\}
\end{equation}
is a smooth submanifold of $\MT\setminus \gamma_v(\mathbb{R})$ and
$
O_q\overset{\rm diff.}{\approx}\pi^{-1}(\pi(q))\overset{\rm diff.}{\approx}SO(2)\overset{\rm diff.}{\approx}\mathbb{S}^1
$. Given $p\in\MT$, $v\in T_p\MT$ and $E_1\in T_p\MT$ with $E_1\perp v$, let us define the totally geodesic half-plane $\Pi_{v,E_1}^+$ by
\begin{equation}
\Pi_{v,E_1}^+:=\{\exp_p(vt_1+E_1t_2)\,:\,t_1\in \mathbb{R}, t_2>0 \}
\end{equation}
Observe that for any $\varphi \in \mathcal{R}_v$,
$$
\varphi(\Pi^+_{v,E_1})=\Pi^+_{v,\varphi_*(E_1)}
$$
because $\varphi$ is an isometry and hence commutes with the exponential map.

\begin{definition}[End of revolution]Given a point $p\in T_p\MT$, two perpendicular vectors $v,w\in T_pM$ of length $1$, and given a curve $\beta:[0,\infty)\to \Pi_{v,w}^+$. An end of revolution $E$ along $\gamma_v$ with generating curve $\beta$ is the set
\begin{equation}
E:=\left\{\varphi(\beta(t))\, :\, t\in [0,\infty)\, {\rm and }\,\varphi\in\mathcal{R}_v\right\}
\end{equation}
\end{definition}

In the following Theorem we shall prove that any end of revolution given by the above definition can be understood as a smooth submanifold with boundary immersed in $\MT$, moreover such an end is intrinsically a warped product.

\begin{theorem}\label{theoprelim}
Given $p\in \MT$, $v\in T_p\MT$. Let $f:\mathbb{S}^1\to\mathcal{R}_v $ be a diffeomorphism. Then, given a smooth and regular curve $\beta:[0,\infty)\to\Pi^+_{v,w}$, for some $w\in T_p\MT$ with $\langle w, v\rangle=0$, the map
\begin{equation}
\alpha:[0,\infty)\times \mathbb{S}^1\to \MT,\quad (t,\theta)\to\alpha(t,\theta)=f(\theta)\big(\beta(t)\big)
\end{equation}
is an immersion. Moreover, there exists a diffeomorphism $\widetilde f:\mathbb{S}_1\to\mathbb{S}_1$  and a positive function $w:\mathbb{R}^1\to\mathbb{R}^1$ such that
\begin{equation}
\alpha^*(g_{\MT})=\Vert \dot\beta\circ\pi_1\Vert^2\pi_1^*g_{\mathbb{R}^1}+\left(w\circ\pi_1\right)^2\left(\widetilde f\circ\pi_2\right)^*g_{\mathbb{S}_1}
\end{equation}
where $\pi_1$ and $\pi_2$ are the projections 
\begin{equation*}
\pi_1:[0,\infty)\times \mathbb{S}_1\to [0,\infty),\quad \pi_2:[0,\infty)\times \mathbb{S}_1\to \mathbb{S}_1
\end{equation*}
and $g_{\mathbb{R}^1}$, $g_{\mathbb{S}_1}$ are the canonical metrics of $\mathbb{R}^1$ and $\mathbb{S}_1$ respectively.
\end{theorem}
\begin{proof}
First of all we have to prove that ${\rm rank}(\alpha_*)=2$. Given $(t_0,\theta_0)\in [0,\infty)\times \mathbb{S}_1$,  the tangent space $T_{(t_0,\theta_0)}[0,\infty)\times \mathbb{S}_1$ can be decomposed as
\begin{equation}\label{decom}
T_{(t_0,\theta_0)}[0,\infty)\times \mathbb{S}_1=T_{t_0}[0,\infty)\, \oplus\, T_{\theta_0}\mathbb{S}_1.
\end{equation}

For any $x\in T_{t_0}\mathbb{R}^1$,  $\alpha_*(x)$ is tangent to the plane $\Pi^+_{v,f(\theta_0)_*(w)}$, because 
\begin{equation}
\begin{aligned}
\alpha_*(x)=&\left.{\frac{d}{dt}}\alpha(t_0+xt,\theta_0)\right\vert_{t=0}=\left.{\frac{d}{dt}}f(\theta_0)\left(\beta(t_0+xt)\right)\right\vert_{t=0}
\end{aligned}
\end{equation}
and since $\beta$ is a curve in $\Pi^+_{v,w}$, then $f(\theta_0)\beta$ is a curve in $\Pi^+_{v,f(\theta_0)_*(w)}$. Moreover,
$$
\alpha_*(x)=f(\theta_0)_*\left(x\dot\beta(t_0)\right)
$$
then $\alpha_*(x)\neq 0$ if $x\neq 0$ because the generating curve $\beta$ is regular ($\dot \beta\neq 0$). When we consider $y\in T_{\theta_0}\mathbb{S}_1$, then $\alpha_*(y)$ is tangent to the orbit space $O_{\alpha(t_0,\theta_0)}$ because for $\gamma:(-\epsilon,\epsilon)\to\mathbb{S}^1$, with $\gamma(0)=\theta_0,\,\dot\gamma(0)=y$,
\begin{equation}
\begin{aligned}
\alpha_*(y)=&\left.{\frac{d}{dt}}\alpha(t_0,\gamma(t))\right\vert_{t=0}=\left.{\frac{d}{dt}}f(\gamma(t))\left(\beta(t_0)\right)\right\vert_{t=0}\in T_{\alpha(\theta_0,t_0)}O_{\alpha(\theta_0,t_0)}.
\end{aligned}
\end{equation}
Moreover, the curve $\gamma$ induces the $1$-parametric subgroup $f(\gamma(t))=\exp(f_*(y)t)f(\theta_0)$ and its action induces a never vanishing vector on $\alpha(t_0,\theta_0)$ whenever $y\neq 0$ because the action of $\mathcal{R}_v$ is freely on $\MT\setminus \gamma_v( \mathbb{R})$ (see proposition 4.1 of \cite{Koba1}). Then $\varphi$ in an immersion.

In order to deal with the induced metric $\varphi^*(g_{\MT})$ observe that  by using the decomposition of equation (\ref{decom}) we have only three cases

\item [{\bfseries Case I: two horizontal vectors}], Suppose $x\in T_{t_0}[0,\infty)$ then, 
\begin{equation}
\begin{aligned}
\alpha^*\left(g_{\MT}\right)(x, x)=&\left\langle \alpha_*(x),\alpha_*(x)\right\rangle_{\MT}\\=&\left\langle f(\theta_0)_*\left(x\dot\beta(t_0)\right),f(\theta_0)_*\left(x\dot\beta(t_0)\right)\right\rangle_{\MT}\\=&\Vert \dot \beta(t_0)\Vert^2\Vert x\Vert^2\\
\end{aligned}
\end{equation}

\item [{\bfseries Case II: one horizontal and one vertical vector}] if $x\in T_{t_0}[0,\infty)$ and $y\in T_{\theta_0}\mathbb{S}_1$, then  
$$\alpha^*\left(g_{\MT}\right)(x,y)=\left\langle \alpha_*(x),\alpha_*(y)\right\rangle_{\MT}=0$$ because 
\begin{proposition}
The orbit space $O_{\alpha(t_0,\theta_0)}$  is perpendicular to $\Pi^+_{v,f(\theta_0)_*w}$.
\end{proposition}
\begin{proof}
Since ${\rm Cut}(p)\cap \left(\MT\setminus \gamma_v\left(\mathbb{R}\right)\right)=\emptyset$, for any $q\in \MT\setminus \gamma_v(\mathbb{R})$ we have a well defined
\begin{equation}
v(q):=\exp_p^{-1}(q).
\end{equation}
To show that $\Pi^{+}_{v,f(\theta_0)_*w}$ is perpendicular to $O_{\alpha(t_0,\theta_0)}$ let us consider the following basis  $\{{\exp_p}_*(v),{\exp_p}_*(W),{\exp_p}_*(\nu)\}$ of $T_{\alpha(t_0,\theta_0)}\MT$, where here we have used $W:=f(\theta_0)_*(w)$ in order to simplify the notation, and $\{v,W,\nu\}$ is an orthonormal basis of $ T_p\MT$ and ${\exp_p}_*$ is the differential of the exponential map, namely
\begin{equation}
\begin{aligned}
{\exp_p}_*(v)&=\left.\frac{d}{dt}\exp_p\left(v(\alpha(t_0,\theta_0))+vt\right)\right\vert_{t=0}\\
{\exp_p}_*(W)&=\left.\frac{d}{dt}\exp_p\left(v(\alpha(t_0,\theta_0))+Wt\right)\right\vert_{t=0}\\
{\exp_p}_*(\nu)&=\left.\frac{d}{dt}\exp_p\left(v(\alpha(t_0,\theta_0))+\nu t\right)\right\vert_{t=0}\\
\end{aligned}
\end{equation} 
hence, ${\exp_p}_*(v)$ and ${\exp_p}_*(W)$ are tangent to $\Pi^{+}_{v,W}$ and moreover, by using the Gauss lemma (see \cite{DoCarmo2}), ${\exp_p}_*(\nu)\perp \Pi^{+}_{v,W}$ and $\{{\exp_p}_*(v),{\exp_p}_*(W),{\exp_p}_*(\nu)\}$ is an orthonormal basis of $T_{\alpha(t_0,\theta_0)}\MT$. Now, let us consider the following two functions
$$
\begin{aligned}
&f_1:\MT\setminus \gamma_v\left(\mathbb{R}\right)\to\mathbb{R},\quad f_1(x)=\langle \exp_p^{-1}(x),\exp_p^{-1}(x)\rangle\\
&f_2:\MT\setminus \gamma_v\left(\mathbb{R}\right)\to\mathbb{R},\quad f_2(x)=\langle v,\exp_p^{-1}(x)\rangle
\end{aligned}
$$
since for any $\varphi\in\mathcal{R}_v$,
 $$
\begin{aligned}
f_1(\varphi(x))=&\langle \exp_p^{-1}\varphi(x),\exp_p^{-1}\varphi(x)\rangle\\=&\langle \varphi_*\exp_p^{-1}(x),\varphi_*\exp_p^{-1}\varphi(x)\rangle=\langle \exp_p^{-1}(x),\exp_p^{-1}(x)\rangle\\
=&f_1(x)\\
{\rm and},&\\
f_2(\varphi(x))=&\langle v,\exp_p^{-1}(\varphi(x))\rangle=\langle v,\varphi_*\exp_p^{-1}(x)\rangle\\=&\langle \varphi_*v,\varphi_*\exp_p^{-1}(x)\rangle=\langle v,\exp_p^{-1}(x)\rangle\\
=&f_2(x)
\end{aligned}
$$
then, $O_{\alpha(t_0,\theta_0)}$ is perpendicular to $\nabla f_1$ and $\nabla f_2$. Now, we are going to show that $\nabla f_1$ and $\nabla f_2$ span $T_{\alpha(t_0,\theta_0)}\Pi^+_{u,W}$. Set $u\in T_p\MT$, and consider the curve 
\begin{equation}
\beta(t)=\exp_p(v(\alpha(t_0,\theta_0))+ut)
\end{equation}
hence, $\beta(0)=\alpha(t_0,\theta_0)$ and $\dot\beta(0)={\exp_p}_*(u)$. Moreover
\begin{equation}
\begin{aligned}
\langle \dot\beta(0),\nabla f_1\rangle=&\frac{d}{dt}\left.f_1\left(\beta\left(t\right)\right)\right\vert_{t=0}=\frac{d}{dt}\left.\left\vert v(\alpha(t_0,\theta_0))+ut\right\vert^2\right\vert_{t=0}=2\langle u, v(\alpha(t_0,\theta_0))\rangle.\\
\langle \dot\beta(0),\nabla f_2\rangle=&\frac{d}{dt}\left.f_2\left(\beta\left(t\right)\right)\right\vert_{t=0}=\frac{d}{dt}\left.\langle v(\alpha(t_0,\theta_0))+ut, v\rangle\right\vert_{t=0}=\langle u, v\rangle.
\end{aligned}
\end{equation}
Then, when we focus on the basis $\{{\exp_p}_*(v),{\exp_p}_*(W),{\exp_p}_*(\nu)\}$ of $T_{\alpha(t_0,\theta_0)}\MT$
\begin{equation}
\begin{aligned}
\begin{pmatrix}
\nabla f_1\\
\nabla f_2
\end{pmatrix}
=\begin{pmatrix}
2\langle v,v(\alpha(t_0,\theta_0))\rangle& 2\langle v(\alpha(t_0,\theta_0)), W\rangle\\
1&0
\end{pmatrix}\cdot
\begin{pmatrix}
{\exp_p}_*(v)\\
{\exp_p}_*(W)
\end{pmatrix}
\end{aligned}
\end{equation}
Finally, Taking into account that $\alpha(t_0,\theta_0)\in \Pi^{+}_{v,W}$ then there exist $t>0$ such that  $\langle v(\alpha(t_0,\theta_0)), W\rangle=t\neq 0$, therefore 
\begin{equation}
\langle\{\nabla f_1,\nabla f_2\}\rangle=\langle\{{\exp_p}_*(v),{\exp_p}_*(W)\}\rangle=T_{\langle v(\alpha(t_0,\theta_0)), W\rangle}\Pi^{+}_{v,W}.
\end{equation}
\end{proof}
\item [{\bfseries Case III: two vertical vectors}.] 
Observe in this case that for every $t_0$ the map $f_1:\mathcal{R}_v\to O_{\beta(t_0)}$ given by $f_1(\theta)=\theta\beta(t_0),
$ 
defines a diffeoromphism from  $\mathcal{R}_v$ to $O_{\beta(t_0)}$. We can pull-back the metric to $O_{\beta(t_0)}$ (an hence to $\mathcal{R}_v$) by using the inclusion $i:O_{\beta(t_0)}$, in such a way that the $1$-dimensional manifold $O_{\beta(t_0)}$ with metric tensor $i^*(g_{\MT})$ is isometric to the $1$-dimensional manifold $\mathcal{R}_v$ with metric tensor $g_R=( i\circ f_1)^*(g_{\MT})$. But we are going to show that the metric $g_R$ is a left invariant metric on $\mathcal{R}_v$. Given $\theta \in \mathcal{R}_v$, let $v,w$ be two vectors of $T_\theta\mathcal{R}_v$, and let $\gamma_1$ and $\gamma_2$ be two curves such that
$$
\gamma_1(0)=\gamma_2(0)=\theta,\quad \dot \gamma_1(0)=v,\quad\dot\gamma_2(0)=w 
$$  
Since $\mathcal{R}_v$ acts by isometries,
$$
\begin{aligned}
g_R({L_\varphi}_* v,{L_\varphi}_* w)_{L_\varphi(\theta)}=&\left\langle \left.\frac{d}{dt}\left(\varphi\gamma_1(t)\right)\right\vert_{t=0},\left.\frac{d}{dt}\left(\varphi\gamma_2(t))\right)\right\vert_{t=0}\right\rangle_{\varphi(\theta (\beta(t_0)))}\\
=&\left\langle\varphi^{-1}_* \left.\frac{d}{dt}(\varphi\gamma_1(t)(q))\right\vert_{t=0},\varphi^{-1}_*\left.\frac{d}{dt}(\varphi\gamma_2(t)(q))\right\vert_{t=0}\right\rangle_{\theta (\beta(t_0))}\\
=&\left\langle\left.\frac{d}{dt}(\varphi^{-1}\varphi\gamma_1(t)(q))\right\vert_{t=0},\frac{d}{dt}(\varphi^{-1}\varphi\gamma_2(t)(q))\vert_{t=0}\right\rangle_{\theta (\beta(t_0))}\\
=&\left\langle\left.\frac{d}{dt}(\gamma_1(t)(q))\right\vert_{t=0},\frac{d}{dt}(\gamma_2(t)(q))\vert_{t=0}\right\rangle_{\theta (\beta(t_0))}\\
=&\langle {f_1}_*(v),{f_1}_*(w)\rangle_{\theta (\beta(t_0))}=g_R(v,w)_\theta.
\end{aligned}
$$ 
But since $\mathcal{R}_v$ is isomorphic (and diffeomorphic) to $SO(2)$ with a left-invariant metric, then $g_R$ is the round metric up to a scale factor. Hence, $\mathcal{R}_v$ is a homotetia of $\mathbb{S}^1$ then
\begin{equation}
\begin{aligned}
\alpha^*\left(g_{\MT}\right)(y,y)=&w^2(t_0)(\widetilde f)^*(g_{\mathbb{S}_1})\left(y,y\right).
\end{aligned}
\end{equation} 
for any $y\in T_{\theta_0}\mathbb{S}_1$.
\end{proof}

\begin{corollary}\label{cor2.9}Under the assumptions of the above Theorem, can be found a local coordinate system $\{t,\theta\}$ of $[0,\infty)\times\mathbb{S}_1$ such that
$$
\alpha^*(g_\MT)=dt^2+w^2(t)d\theta^2.
$$
\end{corollary}

\subsection{Rotationally symmetric model spaces}
Rotationally symmetric model spaces, are generalized manifolds of revolution using warped products. Let us recall here the following definition of a model space.

\begin{definition}[See {\cite{GW, GriExp, GriBook}}]
A $w-$model space $\mathbb{M}_{w}^{n}$ is a simply-connected $n$-dimensional smooth manifold $\mathbb{M}_w^n$ with a point $o_w\in \mathbb{M}_w^n$ called the \emph{center point of the model space} such that $\mathbb{M}_w^n-\{o_w\}$ is isometric to a smooth warped product with base $B^{1}
= (\,0,\, \Lambda)\,\,\subset\, \mathbb{R}$ (where $\, 0 < \Lambda
\leq \infty$\,), fiber $F^{n-1} = \mathbb{S}^{n-1}_{1}$ ({\it i.e.}, the unit
$(n-1)-$sphere with standard metric), and positive warping function $w:\,
(\,0, \,\Lambda\, ) \to \mathbb{R}_{+}$. Namely the metric tensor $g_{\mathbb{M}_w^n}$ is given by: 
\begin{equation}\label{metric-model}
g_{\mathbb{M}_w^n}=r^*\left(g_{(\,0,\, \Lambda)}\right)+(w\circ\pi)^2\Theta^*\left(g_{\mathbb{S}^{n-1}_{1}}\right)\quad,
\end{equation}
$r: \mathbb{M}_w^n\to (\,0,\, \Lambda)$ and $\Theta: \mathbb{M}_w^n\to \mathbb{S}^{n-1}_{1}$ being the projections onto the factors of the warped product, and $g_{(\,0,\, \Lambda)}$ and $g_{\mathbb{S}^{n-1}_{1}}$ the standard metric tensor on the interval and the sphere respectively. 
\end{definition}

 Despite the freedom in the choice of the $w$ function in the above definition, there exist certain restrictions around $r\to 0$. In order to attain $\mathbb{M}_w^n$, a smooth metric tensor around $o_w$, the positive warping function $w$ should hold the following equalities (see \cite{GW,Petersen}):
\begin{equation}
\begin{aligned}
& w(0) = 0\quad,\\
&w'(0) = 1\quad,\\
& w^{(2k)}(0) = 0\quad,
\end{aligned}
\end{equation}
where $w^{(2k)}(r)$ are the even derivatives of $w$. 

The parameter $\Lambda$ in the above definition is called the \emph{radius of the model space}. If
$\Lambda = \infty$, then $o_{w}$ is a pole of $\mathbb{M}_{w}^{n}$ (\emph{i.e.}, the exponential map $\text{exp}_{o_w}:T_{o_w}\mathbb{M}_w^n\to \mathbb{M}_w^n$ is a diffeomorphism).

Observe that a rotationally symmetric model space $\mathbb{M}_w^n$ is  rotationally symmetric at $o_w\in \mathbb{M}^n_w$ in the sense that the isotropy subgroup at $o_w$ of the isometry group is $O(n)$. More commonly, one regards the functions $(r, \Theta)$ as global coordinate functions on $\mathbb{M}_w^n-\{o_w\}$ and the expression of the metric tensor (\ref{metric-model}) is written as
$
g_{\mathbb{M}_w^n}=dr^2+\left(w(r)\right)^2d\Theta^2
$,
where $dr^2$ denotes the standard metric on the interval and $d\Theta^2$ denotes the standard metric on $\mathbb{S}^{n-1}_1$ ( $d\Theta^2=g_{\mathbb{S}^{n-1}_1}$). In this context, $\{r,\Theta\}$ are called \emph{geodesic polar coordinates} around $o_w$.

In view of the definitions of ends of revolution in $\MT$, Corollary \ref{cor2.9} and the definition of a rotationally symmetric model manifold we can state.

\begin{corollary}\label{cor2.11}
Let $E$ be an end of revolution in $\MT$, let $w$ be the warping function given by corollary \ref{cor2.9}. Then for every  positive radius $\rho>0$ there exists a $W\in C^\infty [0,\infty)$  with
$$
W(x)=w(x-\rho),\quad \forall x\geq \rho
$$
such that $E$ is isometric to $\mathbb{M}^2_W\setminus B_{\rho}(o_W)$  where $B_{\rho}(o_W)$ is the geodesic ball of radius $\rho$ centered at $o_W\in \mathbb{M}_W^2$.
\end{corollary}
\begin{proof}
We only have to prove the following lemma
\begin{lemma}\label{lemExt}
Given a positive $\rho>0$ and a positive function $w\in C^\infty[0,\infty)$, the function can be extended to a function $W\in C^{\infty}[0,\infty)$ such that
$$
\left\{
\begin{aligned}
&W(x)=w(x-\rho),\quad \forall x\geq \rho\\
& W(0) = 0\quad,\\
&W'(0) = 1\quad,\\
& W^{(2k)}(0) = 0\quad,\\
&W(x)>0,\quad \forall x>0\quad .
\end{aligned}
\right.
$$  
\end{lemma}
\begin{proof}We can choose $\epsilon>0$ such that $\epsilon<\rho$. We can define moreover the function $F:[\frac{\epsilon}{2},\epsilon]\cup[\rho,\infty)\to\mathbb{R}_+$ given by
$$
F(x):=\left\{
\begin{array}{lcr}
x&{\rm if }&\frac{\epsilon}{2}\leq x\leq\epsilon\\
w(x-\rho)&{\rm if }& x\geq\rho.
\end{array}
\right.
$$ 
Since $F$ is a $C^\infty$ function from the closed set $C:=[\frac{\epsilon}{2},\epsilon]\cup[\rho,\infty)$ of $(0,\infty)$ to $\mathbb{R}_+$ and $F$ has a continuous extension to $(0,\infty)$, then by using the extension lemma for smooth maps (see Corollary 6.27 of \cite{Lee-man}), there exists a $C^\infty$ function $\widetilde F: \mathbb{R}_+\to \mathbb{R}_+$ such that $\widetilde F_{\vert C}=F$. Finally we can define the function $W:[0,\infty)\to [0,\infty)$ by
$$
W(x):=\left\{
\begin{array}{lcr}
x&{\rm if }& x\leq\rho\\
\widetilde F(x)&{\rm if }& x\geq\rho.
\end{array}
\right.
$$ 
\end{proof}
\end{proof}

\subsection{Recurrence and non explosion of ends of revolution}

Conditions for recurrence and non-explosion of the Brownian motion on a Riemannian manifold have been largely studied (see \cite{GriExp} \cite{I1} \cite{I2} \cite{Ahl1}, \cite{Nev} for example). In the particular case of rotationally symmetric model manifolds

\begin{theorem} \label{grig1} \cite{GriExp} Let $\mathbb{M}_w^n$ be a model manifold with $\Lambda=\infty$ (so that $\mathbb{M}_w^n$ is geodesically complete and non-compact). Then $\mathbb{M}_w^n$ is recurrent if and only if
\eq \label{parab}
\int^{\infty} \frac{dt}{w^{n-1}(t)}=\infty
\eeq
\end{theorem}

\begin{theorem} \label{grig2} \cite{GriExp} Let $\mathbb{M}_w^n$ be a model manifold with $\Lambda=\infty$. Then $\mathbb{M}_w^n$ is stochastically complete if and only if
\eq \label{stoc}
\int^{\infty} \frac{\int_0^tw^{n-1}(s)ds}{w^{n-1}(t)} dt=\infty
\eeq
\end{theorem}

Actually (see \cite{Gri2014} proof of Theorem 1.5)  if  $\int^{\infty} \frac{\int_0^tw^{n-1}(s)ds}{w^{n-1}(t)} dt=\infty$ then we can construct  a $1$-superharmonic and radial function $v$ in $M\setminus B_\rho^n(o_w)$ (that is $-\triangle u+u\geq 0$) such that $v(\rho)=1$, $v'(0)=0$, and $v(x)\to\infty$ as $x\to\infty$.

There are sufficient conditions to attain parabolicity for rotationally symmetric model manifolds as well

\begin{theorem}[see \cite{GriExp}]\label{suficient} Let $\mathbb{M}^n_w$ be a rotationally symmetric model manifold. Suppose that
$$
\int^\infty\frac{tdt}{\int_0^tw(s)ds}=\infty,
$$
then $\mathbb{M}^n_w$ is recurrent.
\end{theorem}

\noindent In view of Corollary \ref{cor2.11} we can state

\begin{corollary}\label{cor2.14}Let $E$ be an end of revolution in $\MT$ isometric to the rotationally symmetric model manifold $\mathbb{M}_w^2\setminus B_\rho(o_w)$ for some radius $\rho>0$, then $E$ is parabolic if and only if
\eq \label{parabend}
\int^{\infty}_\rho \frac{dt}{w(t)}=\infty
\eeq
Moreover, if
\eq \label{stocend}
\int^{\infty}_\rho \frac{\int_0^tw(s)ds}{w(t)} dt=\infty
\eeq
then, there exist a compact $K\subset E$ and $1$-superharmonic function such that $v(x)\to\infty$ when $x\to\infty$. 
\end{corollary}

\noindent We will need  moreover the following proposition
\begin{proposition}[see Theorem 1.3 of \cite{Gri2014}]\label{prop2.15}Let $M$ be a connected manifold and $K\subset M$ be a compact set. Assume that there exist a $1$-superharmonic function in $m\setminus K$ such that $v(x)\to\infty$ as $x\to\infty$. Then $M$ is stochastically complete.
\end{proposition}

\subsection{Conformal models of $\MT$ and ends of revolution}\label{confmodel}
The real space forms $\RT$, $\ST$ and $\HT$ can be defined as the 3-dimensional real space forms $\MT$ of constant sectional curvature $\kappa=0,1$ and $-1$ respectively. By using Corollary \ref{cor2.11} each end of revolution is isometric to a rotationally symmetric model manifold $\mathbb{M}_w^2\setminus B_\rho(o_w)$ where we have subtracted some ball $B_\rho(o_w)$. The proof of the Theorems \ref{teounua},\ref{teoendh3},\ref{teoE} makes use of Corollary \ref{cor2.14} were is related the conformal type and the stochastic completeness with the properties of the warping function $w$. Hence, in order to apply Corollary \ref{cor2.14} we need to know the warping function $w$ of such rotationally symmetric models $\mathbb{M}_w^2$. 

From now on we are going to work with conformal models of $\MT$ (see \cite{John}), namely

\eq \label{metricas}
\al
\RT&:= \left\{ (x_1,x_2,x_3) \in \RT\, :\,  g = dx_1^2 +  dx_2^2+ dx_3^2 \right\} \\
\HT& : = \left\{ (x_1,x_2,x_3) \in \RT,\, x_3>0\, :\,  g_{-1} = \frac{1}{x_3^2} \left(dx_1^2 +  dx_2^2+ dx_3^2  \right) \right\}\\
\ST-\{N\}&:= \left\{ (x_1,x_2,x_3) \in \RT\, :\, g_{1} = \frac{4\left(dx_1^2 +  dx_2^2+ dx_3^2  \right)}{\left( 1+x_1^2 + x_2^2 + x_3^2 \right)^2}  \right\}
\aal
\eeq

\vspace{0.5cm}
Indeed, $\MT$  can be seen as $\RT$ endowed with a conformal metric:
\eq \label{conf}
\al
&\MT: = \left\{ (x_1,x_2,x_3) \in \RT :  g_{\kappa} = \eta_{\kappa}(x) \cdot g \right\}
\aal
\eeq
with
\begin{equation}\label{eq3.3}
\eta_{\kappa}(x):=\left\{
\begin{array}{lcl}
\frac{1}{x_3^2}&\text{ if }&\kappa=-1\\
1&\text{ if }&\kappa=0\\
\frac{4}{\left( 1+x_1^2 + x_2^2 + x_3^2 \right)^2}& \text{ if }&\kappa=1
\end{array}
\right.
\end{equation}
Observe that $x_3$-axis in such  models is a geodesic curve. And for any point $x\in\MT$, with  $x=(0,0,x_3)$, 
$$
R_{\theta}: = \left( \begin{array}{ccc}
\cos \theta & -\sin \theta & 0 \\
\sin \theta & \cos \theta & 0 \\
0 & 0 & 1
\end{array} \right)
$$ 
is the subgroup of the isotropy group of $x$ such that the geodesic curve $\gamma(t)=(0,0,x_3+t)$ remains fixed under the action of $R_\theta$. Taking into account moreover that the $(x_1,x_3)$-plane is a totally geodesic submanifold we can construct ends of revolution in $\MT$  in the following way:

\begin{definition} \label{rev} 
 Let $\Gamma$ be a regular curve in the $(x_1,x_3)$-plane. We define an end of revolution $E:=R_{\theta}\Gamma $ as the action of $R_{\theta}$ to curve $\Gamma$. Here we parameterize the curve $\Gamma$ as
$$
\Gamma:= \gamma(s) = (\gamma_1(s), 0, \gamma_2(s))
$$
The end of revolution can be parametrized therefore as
\eq  \label{immersion}
E:= f(s, \theta) = \left( \gamma_1(s) \cos \theta, \gamma_1(s) \sin \theta, \gamma_2(s) \right)
\eeq
\end{definition}

\begin{remark} Note that by using Theorem \ref{theoprelim} if we impose $\gamma_1(s)>0$, the map given in formula (\ref{immersion}) can be considered as an immersion from $[0,\infty)\times\mathbb{S}_1$ to $\MT$. \end{remark}

Lets recall that every regular curve admits a reparametrization by arc length. Then, from the metrics defined in (\ref{metricas}) and considering that the parametrization of the immersion given in (\ref{immersion}) satisfies $\| f_s(s, \theta) \|^2 = 1$, it is easy to see that the metric inherited by the end from each ambient space $\MT$ can be written as 
\eq \label{metric}
g_E = ds^2 + w_\kappa^2(s) d\theta^2
\eeq
where
\begin{equation}\label{eq2.58}
w_{\kappa}(s):=\left\{
\begin{array}{lcl}
\frac{\gamma_1(s)}{\gamma_2(s)}&\text{ if }&\kappa=-1\\
\gamma_1(s)&\text{ if }&\kappa=0\\
\frac{2 \gamma_1(s)}{ 1+ \gamma_1^2(s) + \gamma_2^2(s)}& \text{ if }&\kappa=1
\end{array}
\right.
\end{equation}
Indeed, taking $\eta_{\kappa}(s)$ from definition (\ref{eq3.3}) we can rewrite the inherited metric as
\eq \label{metric2}
g_E= ds^2 + \eta_{\kappa}^2 (s) w_0^2(s) d\theta^2
\eeq
Hence we can summarize with the following proposition
\begin{proposition}\label{prop2.16} Let $E$ be an end of revolution of $\MT$ parametrizated by
$$f(s, \theta) = \left( \gamma_1(s) \cos \theta, \gamma_1(s) \sin \theta, \gamma_2(s) \right)$$ with $s\in [0,\infty)$, $\theta\in [0,2\pi]$. Suppose $\gamma(s)=\left( \gamma_1(s), 0, \gamma_2(s) \right)$ is a regular curve parametrized by arc length and suppose moreover that $\gamma_1>0$. Then, for  $\rho>0$, $E$ is isometric to $\mathbb{M}_{w}^2\setminus B_\rho(o_w)$, where  $\mathbb{M}_{w}^2$ is the rotationally symmetric model space given by  warping function $w$ satisfying
\begin{equation}
\left\{
\begin{array}{lcl}
w(t+\rho)=w_\kappa(t)&\text{ for }&t\geq 0\\
w(0)=0& &\\
w'(0)=0& &\\
w^{(2k)}(0)=0& &\\
\end{array}
\right.
\end{equation}
with $w_\kappa$ given in definition (\ref{eq2.58}). Hence by applying corollary \ref{cor2.14} $E$ is parabolic if and only if
\eq 
\int^{\infty}_\rho \frac{dt}{w(t)}=\int^{\infty}_0 \frac{ds}{w_\kappa(s)}=\infty
\eeq
and if
\eq 
\int^{\infty}_\rho \frac{\int_0^tw(s)ds}{w(t)} dt=\int_0^\infty \frac{\int_0^\rho w(s)ds+\int_0^tw_\kappa(s)ds}{w_{\kappa}(t)}dt=\infty
\eeq 
then, there exist a compact set $K\subset E$ and $1$-superharmonic function such that $v(x)\to\infty$ when $x\to\infty$.
\end{proposition}
From Theorem \ref{suficient} we can state
\begin{corollary}\label{suf2}Let $E$ be an end of revolution in $\MT$  if
$$
\int^\infty\frac{tdt}{\int_0^tw_\kappa(s)ds}=\infty,
$$
then $E$ is parabolic.
\end{corollary}

\section{Proof of Theorem \ref{teounua}}\label{proofA}
Theorem \ref{teounua} states that any end of revolution in $\RT$ or in $\ST$ is a parabolic end. Here we split the prove in these two settings 
\subsection{End immersed in $\RT$}
\begin{proof} 
Lets recall that the  generating curve $\gamma(s)$ was parameterized by its arc length. This implies that $(\dot\gamma_1)^2(s) \le 1$. Using definition (\ref{eq2.58}) we find this equivalent to
$$
-1 \le \dot w_0(t) \le 1
$$
Integrating the latter we obtain
$$
-t  \le w_0(t)-w_0(0) \le t
$$
and thus $w_0(t) \le t + w_0(0)$. By using the criterion for parabolicity given by proposition \ref{prop2.16}
\begin{equation}
\int^{\infty}_0 \frac{1}{w_0(s)}ds = \lim_{R \to \infty} \int_{0}^R \frac{1}{w_0(s)}ds \ge \lim_{R \to \infty} \int_{0}^{R} \frac{1}{t+w_0(0)} dt = \infty.
\end{equation}
Which means that each complete end of revolution $E$, when immersed in $\RT$, is of parabolic conformal type independently of the curve $\gamma(s)$. 
\end{proof}
\subsection{End immersed in $\ST$}
\begin{proof} 
Applying the criterion for parabolicity at proposition \ref{prop2.16}  and the expression (\ref{eq2.58}) for the function $w_1(s)$ we have to prove that the following integral is divergent
$$
\int^{\infty}_0 \frac{1}{w_1(s)} ds  = \int^{\infty}_0 \frac{1+\gamma_1^2(s) + \gamma_2^2(s) }{2 \gamma_1(s)} ds.
$$
For any $\varepsilon>0$, we can split the interval where we integrate ($I =[0, \infty)$) in two parts: $I_+ =\{t \in I : \gamma_1(t) \ge \varepsilon \}$ and $I_-=\{ t \in I : \gamma_1(t) < \varepsilon \}$, such that $I_+ \cup I_- = I$, $I_+ \cap I_- = \emptyset$ and since $\int_I dx=\infty$, thus $\int_{I_+}dx + \int_{I_-}dx= \infty$. Then we have two cases.

\item [{\bfseries Case I:} $\int_{I_+}dx = \infty$], so as we have seen:
$$
\int_{I} \frac{1}{w_1(s)} ds \ge \int_{I_+} \frac{1+\gamma_1^2(s) + \gamma_2^2(s) }{2 \gamma_1(s)} ds \ge \frac{1}{2} \int_{I_+} \gamma_1(s) ds \ge \frac{1}{2}\varepsilon \int_{I_+} ds = \infty
$$

\item [{\bfseries Case II:} $\int_{I_+}dx <\infty$ ($\int_{I_-}dx = \infty$)], then
$$
\al
\int_{I} \frac{1}{w_1(s)} ds \geq &\int_{I_-} \frac{1+\gamma_1^2(s) + \gamma_2^2(s) }{2 \gamma_1(s)} ds \ge  \int_{I_-} \frac{1}{2 \gamma_1(s)} ds\\
 \ge &\int_{I_-} \frac{1}{2 \varepsilon} ds \ge \frac{1}{2 \varepsilon} \int_{I_-} ds = \infty
\aal
$$

Which means again that the end is parabolic independently of the curve $\gamma(s)$.\end{proof}
\section{Proof Theorem \ref{teoendh3} } \label{proofB}

Theorem \ref{teoendh3} states that every end of revolution on a $c$-cone is a parabolic end. Observe that if the end is on a $c$-cone then the generating profile curve  $\gamma(s)=(\gamma_1(s),0,\gamma_2(s))$ satisfies 
\begin{equation}
\frac{\gamma_2(s)}{\gamma_1(s)}\geq c.
\end{equation}
 Substituting the function $w_{-1}(s)$ given by (\ref{eq2.58}) in the criterion for parabolicity given in proposition (\ref{prop2.16}), we get that the end of the surface will be parabolic because
$$
\int^{\infty}_0 \frac{1}{w_{-1}(s)}ds = \int^{\infty}_0 \frac{\gamma_2(s)}{\gamma_1(s)} ds \geq\int^{\infty}_0 c\,ds= \infty.
$$
This finishes the proof of Theorem \ref{teoendh3}. However, we can state something more general. let us denote $I_+ := \{t \in I : \frac{\gamma_2(t)}{\gamma_1(t)} \ge c \}$ and $I_- := \{t \in I : \frac{\gamma_2(t)}{\gamma_1(t)} < c\}$. Then,
$$
\int_I \frac{1}{w_{-1}(s)}ds = \int_{I_+} \frac{\gamma_2(s)}{\gamma_1(s)} ds +  \int_{I_-} \frac{\gamma_2(s)}{\gamma_1(s)} ds\geq c\int_{I_+}ds +  \int_{I_-} \frac{\gamma_2(s)}{\gamma_1(s)}ds.
$$
Hence if $\int_{I_+}ds =\infty$, the end will still be parabolic. Therefore
\begin{theorem}\label{teo4.1}
Let $E$ be an end of revolution in $\HT$. Suppose that the generating curve of $E$ satisfies
$$
\int_{I_+}ds=\infty.
$$
Then, the end is parabolic.
\end{theorem}

\section{Proof of Theorem \ref{slice}}\label{proofC}
The first step to prove Theorem \ref{slice} is to prove previously the following proposition

\begin{proposition}\label{limitedx}Let $E$ be a complete end of revolution immersed in $\HT$. Suppose that $E$ is a non-parabolic end, and $\gamma:[0,\infty)\to \HT$ is the profile curve of $E$ parametrized by arc length in the half space model of the hyperbolic space given by
$$
\gamma(s)=(\gamma_1(s),0,\gamma_2(s))
$$
then 
$$
\sup_{s\in [0,\infty)} \gamma_1(s)<\infty,
$$
\end{proposition}
\begin{proof}
Since $\gamma_1$ is a positive and smooth function, then,
\begin{equation}\label{limited0}
\log\gamma_1(s)-\log\gamma_1(0)=\int_0^s\frac{d}{dt}\left(\log \gamma_1(t)\right)dt=\int_0^s\frac{\dot\gamma_1(t)}{\gamma_1(t)}dt.
\end{equation}
But taking into account that $\gamma$ is parametrizated by arc length, namely, $$\frac{(\dot\gamma_1(s))^2+(\dot\gamma_2(s))^2}{(\gamma_2(s))^2}=1,$$ then,
$$
\dot \gamma_1(s)\leq \vert \dot \gamma_1(s)\vert \leq \gamma_2(s)
$$
and hence, inequality (\ref{limited0}) can be rewritten as
\begin{equation}
\log\gamma_1(s)-\log\gamma_1(0)\leq\int_0^s\frac{\gamma_2(t)}{\gamma_1(t)}dt
\end{equation}
By using the function $w_{-1}(s)$ given by (\ref{eq2.58}) in the criterion for parabolicity given in proposition \ref{prop2.16}
\begin{equation}
\int^{\infty}_0 \frac{1}{w_{-1}(s)}ds = \int^{\infty}_0 \frac{\gamma_2(s)}{\gamma_1(s)} ds <\infty.
\end{equation}
And hence,
\begin{equation}
\log\gamma_1(s)-\log\gamma_1(0)<\int_0^\infty\frac{\gamma_2(t)}{\gamma_1(t)}dt<\infty. 
\end{equation}
\end{proof}
\begin{proof}[Proof of Theorem \ref{slice}]
Theorem \ref{slice} states that any complete end of revolution contained on a horosphere $\{x_3=z\}$ is a parabolic end. Since the end is on the horosphere $\{x_3=z\}$
then 
\begin{equation}\label{eqslice}
\int^{\infty}_0 \frac{1}{w_{-1}(s)}ds = \int^{\infty}_0 \frac{\gamma_2(s)}{\gamma_1(s)} ds \geq z\int^{\infty}_0 \frac{1}{\gamma_1(s)} ds.
\end{equation}
Hence the end is parabolic because otherwise if we suppose  that $E$ is non-parabolic, by using proposition \ref{limitedx},
\begin{equation}
\int^{\infty}_0 \frac{1}{w_{-1}(s)}ds \geq z\int^{\infty}_0 \frac{1}{\sup_{s\in[0,\infty)}\gamma_1(s)}ds=\infty.
\end{equation}
Then by using the criterion for parabolicity given in proposition \ref{prop2.16} the end $E$ is parabolic (contradiction).
\end{proof}
\section{Proof of Theorem \ref{teoE}}\label{proofE}

Recall that Theorem \ref{teoE} states
\begin{thm}
 Let $\Sigma$ be complete and non-compact surface of finite topological type immersed in $\MT$ with $\kappa\in\mathbb{R}$. Suppose that there exists a compact subset $\Omega\subset \Sigma$ of $\Sigma$ such that every end of $\Sigma$ with respect to $\Omega$ is an end of revolution  in $\MT$. Then, $\Sigma$  is stochastically complete.
\end{thm}
\begin{proof}
When $\Sigma$ is immersed in $\RT$ or in $\ST$ the stochastic completeness of its ends is straight forward according that every parabolic surface  is stochastically complete (see \cite{GriExp}for instance). For surfaces $\Sigma$ in $\HT$ such that every of its ends with respect to some compact $\Omega\subset \Sigma$ is an end of revolution in $\HT$, we are going to show that there exist a $1$-superharmonic function satisfying the hypothesis of proposition \ref{prop2.15} (and hence $\Sigma$ is stochastically complete). In order to construct such a function  we are going to use the following proposition 
\begin{proposition}Let $E$ be an end of revolution in $\HT$, then, there exist a compact set $K\subset E$ and $1$-superharmonic function such that $v(x)\to\infty$ when $x\to\infty$.
\end{proposition}
\begin{proof}
From (\ref{eq2.58}) we have that 
$$
w_{-1}(s) = \frac{\gamma_1(s)}{\gamma_2(s)}
$$
hence,
$$
\frac{\dot w_{-1}(s)}{w_{-1}(s)} = \frac{\dot \gamma_1(s)}{\gamma_1(s)} - \frac{\dot\gamma_2(s)}{\gamma_2(s)}
$$
From the condition that $\gamma(s)$ is parameterized by arc length we have that $\vert \dot\gamma_1(s) \vert \le \gamma_2(s)$ and $\vert \dot\gamma_2(s) \vert \le \gamma_2(s)$. Then

\begin{equation}\label{eq5.1}
\frac{\dot w_{-1}(s)}{w_{-1}(s)}\leq \frac{\gamma_2(s)}{\gamma_1(s)} + \frac{\gamma_2(s)}{\gamma_2(s)} = \frac{1}{w_{-1}(s)} + 1
\end{equation}

For any $c>0$ we can now split the interval $I=[0,\infty)$ in two parts, $I = I_+ \cup I_-$ such that $I_+:=\{s \in \mathbb{R} : \frac{\gamma_2(s)}{\gamma_1(s)} \ge c\}$ and $I_-=\{s \in \mathbb{R} : \frac{\gamma_2(s)}{\gamma_1(s)} < c \}$. With $I_+ \cap I_- = \emptyset$, $\int_Ids = \int_{I_+}ds + \int_{I_-}ds = \infty$.We have again two cases:

\item[{\bfseries Case I:} $\int_{I_+}ds=\infty$.] Then
$$
\al
\int^{\infty} \frac{1}{w_{-1}(s)}ds & = \int_{I_+} \frac{1}{w_{-1}(s)}ds + \int_{I_-} \frac{1}{w_{-1}(s)}ds \geq  \int_{I_+} c ds +  \int_{I_-} \frac{1}{w_{-1}(s)}ds \\
& = c \cdot \int_{ I_+}ds +  \int_{I_-} \frac{1}{w_{-1}(s)}ds = \infty
\aal
$$
Hence, by applying corollary \ref{cor2.14} the proposition follows.
\item[{\bfseries Case II:} $\int_{ I_+}ds < \infty$ ($\int_{ I_-}ds = \infty$).] for $s>\rho$,

\begin{equation}
\al
w(s)-w(\rho) &= \int_{\rho}^s \dot w(r)dr = \int_{[\rho, s] \cap I_+}\dot w(r)dr + \int_{[\rho,s]\cap I_-} \dot w(r)dr\\
 & < \int_{[\rho, s] \cap I_+}w'(r)dr + (1+c) \int_{[\rho,s]\cap I_-} w(r)dr\\
\aal
\end{equation}

where we have considered that in $I_-$ by using inequality (\ref{eq5.1}), $\dot w(r) < (1+c) \, w(r)$. Then,

$$
1 < \frac{w(\rho)}{w(s)} + \frac{ \int_{[\rho, s] \cap I_+}w'(r)dr}{w(s)} + \frac{(1+c) \int_{[\rho,s]\cap I_-} w(r)dr}{w(s)}.
$$
Integrating in $[\rho,R]$,
\begin{equation}
\al
R-\rho &< \int^{R}_\rho\frac{w(\rho)}{w(s)} ds + \int^{R}_\rho \frac{ \int_{[\rho, s] \cap I_+}w'(r)dr}{w(s)} ds + \int^{R}_\rho \frac{(1+c)\int_{[\rho,s]\cap I_-} w(r)dr}{w(s)}ds\\
& \le \int^{\infty}_\rho\frac{w(\rho)}{w(s)} ds + \int^{\infty}_\rho \frac{ \int_{[\rho, s] \cap I_+}w'(r)dr}{w(s)} ds + (1+c)  \int^{\infty}_\rho  \frac{ \int_{\rho}^s w(r)dr}{w(s)}ds
\aal
\end{equation}
taking into account that  by using inequality (\ref{eq5.1}), $w'(r)\leq 1 + w(r) $, then for any $R>\rho$,
\begin{equation}
\begin{aligned}
R-\rho  < & w(\rho)\int^{\infty}_\rho\frac{1}{w(s)} ds + \int^{\infty}_\rho \frac{ \int_{[\rho, s] \cap I_+}(1+w(r))dr}{w(s)} ds\\
&+ (1+c) \int^{\infty}_\rho  \frac{ \int_{\rho}^s w(r)dr}{w(s)}ds\\
 \le & \left(w(\rho)+\int_{I_+}dr\right)\int^{\infty}\frac{1}{w(s)} ds + (2+c) \int^{\infty}_\rho  \frac{ \int_{\rho}^s w(r)dr}{w(s)}ds.
\end{aligned}
\end{equation}
Letting now $R$ tent to infinity we conclude that or
$$
\int^{\infty}_0\frac{1}{w(s)} ds=\infty
$$
or
$$
\int^{\infty}_\rho  \frac{ \int_{\rho}^s w(r)dr}{w(s)}ds.
$$
But in any case the proposition follows by using proposition \ref{prop2.16}.
\end{proof}
Now by applying the above proposition in each connected component of $\Sigma\setminus \Omega$ (call them $\{E_i\}$) there exist a compact set $K_i\subset E_i$ and a $1$-superharmonic function $v_i$ in $E_i\setminus K_i$ such that $v_i(x)\to\infty$ as $x\to\infty$. Defining a compact $C$ as
$$
C=\Omega\cup_i K_i
$$
and the function $\widetilde v:\Sigma\setminus C\to \mathbb{R}$ by
$$
\widetilde v(x):=v_i(x),\quad \text{ if }x\in E_i.
$$
we conclude that $\widetilde v$ is $1$-superharmonic function in $\Sigma\setminus C$ and $\widetilde v(x)\to\infty$ as $x\to\infty$. Hence, by applying proposition \ref{prop2.16} the Corollary is proved.
\end{proof}

\section{Proof of Theorem \ref{horoteo}}\label{proofF}

Recently in \cite{Pacelli2013} is proved that any non-compact surface $\Sigma$ which is stochastically and geodesically complete,  properly immersed into a horoball of the hyperbolic space $B\subset \HT$ has  
\begin{equation}
\sup_{x\in \Sigma}\Vert \vec H\Vert\geq 1,
\end{equation}
and
\begin{equation}
\sup_{x\in \Sigma}K_G\geq 0.
\end{equation}

Hence, there are no surfaces of revolution which are negatively curved and properly immersed into an horoball (statement (1) of Theorem \ref{horoteo}). Moreover, if $\Sigma$ is a cmc-surface, $\Sigma$ is a cmc one surface and therefore by using \cite{Rod98} is a horosphere. On the other hand if $\Sigma$ has constant non-positive sectional curvature, then $K_G=0$. But the only complete flat surface contained in a horoball is the horosphere (see Theorem 3 of \cite{Gal2000} for instance). 
\section{Movement of the centroid of a curve in $\mathbb{H}^2$ and its applications to the conformal type}

Given a regular curve $\gamma:[0,\infty)\to\mathbb{H}^2\subset \HT$, parametrized by arc length in the half space model of the hyperbolic space $$\gamma(s)=(\gamma_1(s),0,\gamma_2(s)),\quad {\rm with}\quad \gamma_1>0\text{ and }\gamma_2>0$$ We shall say that the segment $\gamma([0,s])$ has  \emph{centroid} $x_ g(\gamma([0,s]))$, given by
$$
x_g(\gamma([0,s])):=\frac{\int_{0}^s w_{-1}(t)dt}{s}
$$

\begin{theorem}\label{teo6.1} Suppose that the centroid of a regular curve $\gamma:[0,\infty)\to\mathbb{H}^2$ satisfies 
$$
x_g(\gamma[0,s])\leq C\, s
$$ for some $C>0$ and any $s\geq s_0$ for some $s_0>0$. Then the end of revolution given by
$$
f(\theta,s):=R_\theta \gamma
$$
is a parabolic end. \end{theorem}
\begin{proof}Observe that
\begin{equation}
\begin{aligned}
\int^\infty\frac{tdt}{\int_0^tw(s)ds}&\geq\int^\infty_{s_0}\frac{dt}{x_g(\gamma([0,t]))}\geq \int^\infty\frac{dt}{Ct}=\infty.
\end{aligned}
\end{equation}
And the Theorem is proved by using corollary \ref{suf2}.
\end{proof}

\begin{definition}
We shall say  that a regular curve $\gamma:[0,\infty)\to\mathbb{H}^2$, parametrized by arc length $\gamma(s)=(\gamma_1(s),0,\gamma_2(s))$,  (where we have used the half space model of the hyperbolic space and we have assumed $\gamma_1>0$, $\gamma_2>0$),  has \emph{confined centroid} if the limit $\lim_{s\to\infty}x_g(\gamma([0,s]))$ exists and
$$
\lim_{s\to\infty}x_g(\gamma([0,s]))<\infty.
$$ 
\end{definition}

\begin{theorem} Suppose that f a regular curve $\gamma:[0,\infty)\to\mathbb{H}^2$ has confined centroid. Then the end of revolution given by
$$
f(\theta,s):=R_\theta \gamma
$$
is a parabolic end. \end{theorem}

\begin{proof} Since $\gamma$ has confined centroid, for each $\epsilon>0$ there exist $R$ large enough such that for any $t>R$
\begin{equation}
\frac{\int_{0}^t w_{-1}(s)ds}{t}\leq x_g(\gamma)+\epsilon.
\end{equation} 
Therefore
\begin{equation}
\begin{aligned}
\int_R^\infty\frac{tdt}{\int_0^tw(s)ds}&\geq\int_R^\infty\frac{tdt}{(x_g(\gamma)+\epsilon)t}=\infty.
\end{aligned}
\end{equation}
And the Theorem is proved by using corollary \ref{suf2}.
\end{proof}

\section{Examples of application} \label{ex}

We would like to highlight some examples with relevant properties of complete ends of revolution in $\MT$.

\subsection{Surfaces immersed into a ball of $\RT$}

The topic of complete immersions into a geodesic balls of $\RT$ has been largely studied from the Labyrinth example of Nadirashvili (cf. \cite{Nadi}). From the conformally point of view the Brownian motion of any complete bounded minimal surface in $\RT$ is transcient (non-recurrent)  (see \cite{PacMon} for instance). Moreover, the Brownian movement of a submanifold is transcient (see \cite{GJellet}) if the submanifold admits a complete immersion within a geodesic ball of radius $R$ with mean curvature vector field $\vec H$ bounded by
$$
\Vert \vec H \Vert <\frac{1}{R}
$$ 
Taking into account that by Theorem \ref{teounua} any end of revolution in $\RT$ is a parabolic end we can state
\begin{corollary}\label{cor7.1}
Let $\Sigma$ be a surface isometrically immersed into a geodesic ball $B_R\subset \RT$. Suppose that $\Sigma$ has at least one end of revolution in $\RT$. Then, the mean curvature vector field $H$ satisfies 
$$
\sup_{x\in \Sigma}\Vert \vec H(x) \Vert \geq \frac{1}{R}.
$$ 
\end{corollary}
\begin{remark}
Jorge and Xavier proved in \cite{JorgeXavier81},  that every submanifold $M$ whose scalar curvature is bounded from below immersed in a geodesic ball $B_R\subset \mathbb{R}^n$ of radius $R$ satisfies
$$
\sup_{M}\Vert \vec H\Vert\geq \frac{1}{R}
$$
\end{remark}
\begin{proof}[Proof of Corollary \ref{cor7.1}]
Suppose by the contrary that 
$$
\sup_{x\in \Sigma}\Vert \vec H(x) \Vert < \frac{1}{R}.
$$ 
Then, by using Corollary 2.7 of \cite{GJellet}, $\Sigma$ has positive Cheeger constant $h(\Sigma)>0$, in particular the end of revolution $E$ has also positive Cheeger constant $h(E)>0$, and therefore $E$ has positive fundamental tone $\lambda^*(E)$ which implies that $E$ is non-parabolic (see \cite{GriExp}) in contradiction with Theorem \ref{teounua}.
\end{proof}
In the particular case of minimal surfaces, Corollary \ref{cor7.1} implies that there does not exist bounded surfaces of revolution. Actually, that agrees with the classical result of Bonnet which states that, up to rigid motions, in $\RT$ the only minimal surfaces of revolution are the catenoid and the plane.

A natural question is to ask for the existence of recurrent surfaces immersed into a geodesic ball of $\RT$. We can guarantee the existence of such a surfaces and it can be seen through the following \emph{self made} example. 
\begin{example}
Consider the curve $\Gamma: \mathbb{R} \to \mathbb{R}^3$ parametrized as (see also figure \ref{fig:example1}):
$$
\alpha_{R,a}(t) = \left( \frac{(R-a)t^2}{(t^2+1)}+a, 0,\sin \left( \frac{(R-a)t^3}{(t^2+1)}+at \right) \right)
$$

When it is rotated over the $(0,0,x_3)$ axis, it generates a complete surface of revolution in $\RT$ which is bounded, i.e. it can be kept inside a cylinder $\mathbb{R} \times \mathbb{S}^1$ with radius $R$ and height $h=1$. By using Theorem \ref{teounua} this surface is recurrent. Note also that the mean curvature of this surface is unbounded.

\begin{figure} 
   \centering
   \includegraphics[width=1.2in]{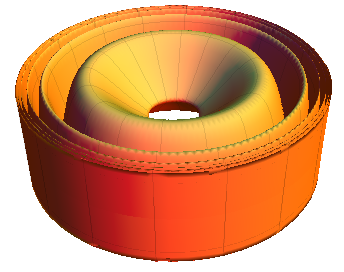}
   \includegraphics[width=1.3in]{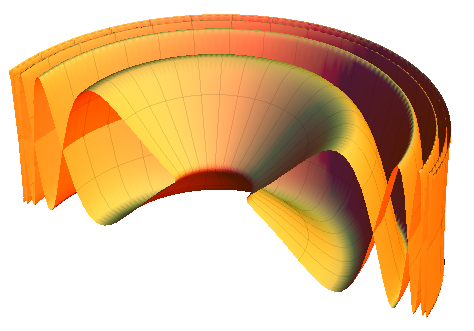}
   \includegraphics[width=1.3in]{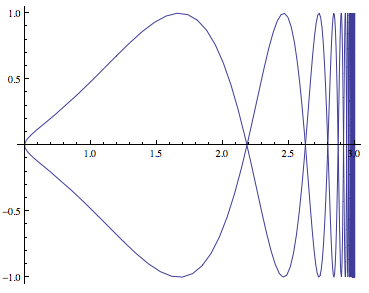}   
   \caption{A parabolic surface in a ball of $\RT$ and its generating curve $\alpha_{R,a}(t) $.}
   \label{fig:example1}
\end{figure}

\end{example}

\vspace{0.5cm}

\subsection{Surfaces in $\HT$ with transcient Brownian movement} The spherical catenoids  immersed in $\HT$ are example of surfaces of revolution in $\HT$ with transcient (non-recurrent) Brownian movement. Spherical catenoids have been studied in \cite{DoDa}, \cite{Mo} or \cite{Seo} and, specifically using the Upper Halfspace Model in \cite{berard2010}. 

\begin{example}[Spherical catenoids]
Spherical catenoids are the minimal complete surfaces of revolution generated by the rotation of the family of curves
$$
\gamma_a(s) = \left( e^{\Lambda_a(s)} \tanh(y_a(s)), 0,\frac{e^{\Lambda_a(s)}}{\cosh(y_a(s))} \right) 
$$ 
where 
$$
y_a(s):=a+\int_0^s\frac{\cosh(2a)\sinh(2t)}{(\cosh(2a)^2\cosh(2t)^2-1)^{\frac{1}{2}}}dt.
$$
and
$$
\Lambda_a(s):=\sqrt{2}\sinh(2a)\int_0^s\frac{(\cosh(2a)\cosh(2t)-1)^{\frac{1}{2}}}{\cosh^2(2a)\cosh^2(2t)-1}dt.
$$
The warping function is thus given by
$$
w_{-1}(s) = \frac{\gamma_{a1}(s)}{\gamma_{a2}(s)}  = \sinh(y_a(s))
$$
and hence the following integral can be showed as a divergent integral.
$$
\int^{\infty} \frac{1}{w_{-1}(s)} ds = \int^{\infty} \frac{1}{\sinh(y_a(s))}ds < \infty
$$
proving that each one of the ends of the surface is non-parabolic. However, transcience of spherical catenoids can be proved taking into account that spherical catenoids are minimal surfaces, and every minimal surface of $\HT$ is transcient (see \cite{MPcapacity} for instance).

Observe (see figure \ref{fig:example2})  that by Theorem \ref{teo4.1}, the part of the curve $\gamma_a(s)$ lying over an arbitrary line $t \to (t, c \cdot t)$ has to be of finite length.

\begin{figure} 
   \centering
   \includegraphics[width=1.5 in]{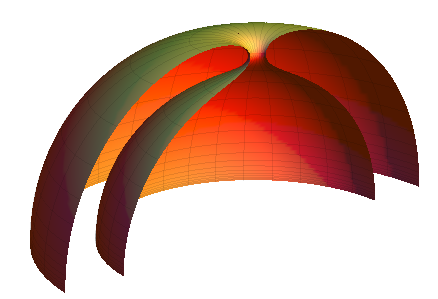} \hspace{0.5cm}
   \includegraphics[width=1.5 in]{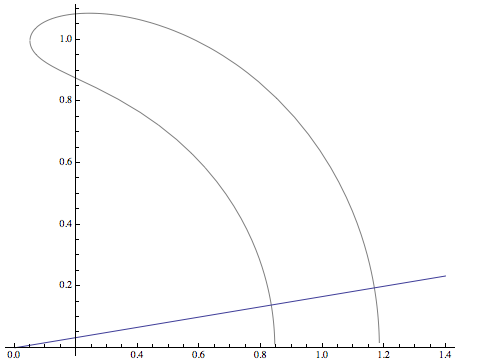} \\ \vspace{0.5cm}
   \includegraphics[width=1.7 in]{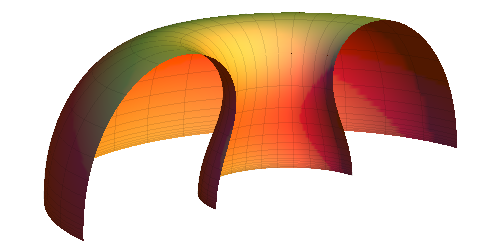}  \hspace{0.5cm}
   \includegraphics[width=1.7 in]{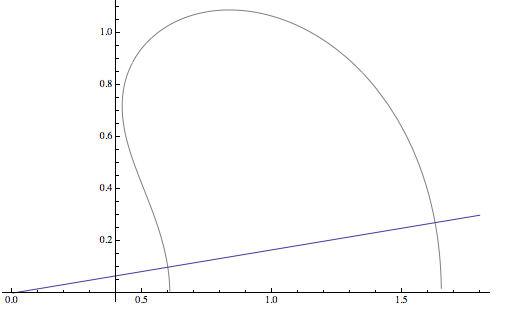} 
   \caption{Half \emph{Spherical Catenoids} immersed in $\HT$: $\Sigma_{0.05}$ and $\Sigma_{0.5}$. Observe that the catenoid lies under a horosphere, the part of the profile curve lying over an arbitrary line $t \to (t,0, c \cdot t)$ has finite length (Theorem \ref{teo4.1}), each one of the two ends approaches to the plane $\{x_3=0\}$ (Corollary \ref{corendhyp}), and $\sup x_1<\infty$ (Proposition \ref{limitedx}).}
   \label{fig:example2}
\end{figure}
\end{example}
\subsection{Surfaces in $\HT$ with recurrent Brownian motion}To construct a surface of revolution in $\HT$ with recurrent Brownian motion can be achieved, by using our Theorem \ref{teoendh3},  if we construct a surface such that every end is of revolution and every end is on a $c$-cone for some $c>0$. In our example we are using clothoids. 
\begin{example}
The \emph{clothoids} or \emph{spirals of Cornu} are curves generated by pairs of functions of the form
$$
\text{clothoid}[n,a][t]= a \left( \int_0^t \sin \left(\frac{s^{n+1}}{n+1} \right)ds,0, \int_0^t \cos \left( \frac{s^{n+1}}{n+1} \right)ds \right)
$$
and commonly used in construction (cf. \cite{Gray}). By changing $t \to e^t$, in case $a=n=1$, we obtain a complete curve which can be easily immersed in $\mathbb{H}^2$. The surface of revolution obtained when rotating the curve over the vertical axes appears to have two parabolic ends (see figure \ref{fig:example3}). Note that the immersion of the surface is not proper.

\begin{figure} 
   \centering
   \includegraphics[scale=0.25]{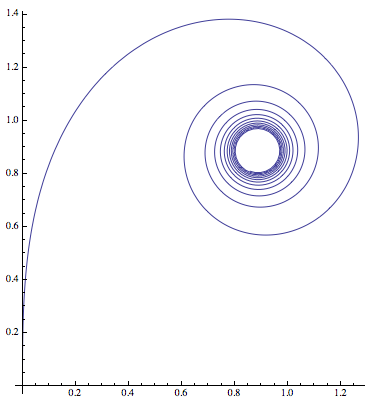} 
        \includegraphics[scale=0.35]{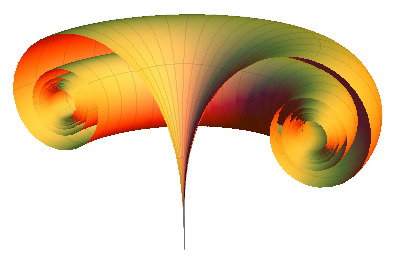} \\
   \caption{Clothoid or spiral of Cornu: the curve, the curve reparameterized, and the rotation with parameters $a=1, n=1$.}
   \label{fig:example3}
\end{figure}
\end{example}
\begin{example}[horosphere]
An other interesting example of surface of revolution is the horosphere which in the upper half space model of $\HT$ is just the $x_3=z$ plane. An end of revolution can be obtained rotating the parametrized by arc length curve
$$
\gamma:[0,\infty)\to \HT,\quad \gamma(s)=(z\, s+1,0,z)
$$  
along the $x_3$-axis. Observe that
$$
w_{-1}(s)=\frac{\gamma_1(s)}{\gamma_2(s)}=s+\frac{1}{z}.
$$
If we observe the movement of the centroid
$$
x_g(\gamma([0,s])=\frac{\int_0^sw_{-1}(s)ds}{s}=\frac{s}{2}+\frac{1}{z}
$$
Hence, given $s_0>0$, for any $s\geq s_0$ and denoting $C:=(\frac{1}{2}+\frac{1}{zs_0})$,
$$
x_g(\gamma([0,s])=s(\frac{1}{2}+\frac{1}{zs})\leq C\, s
$$
By Theorem \ref{teo6.1} the horosphere is a recurrent surface. This result can be achieved directly by using Theorem \ref{slice} or by the fact that since the horosphere is a flat surface, it has finite total curvature and hence (by using \cite{I1}) the Brownian movement is recurrent.
\end{example}
\subsection{Surfaces of revolution in $\HT$ with one parabolic end and one non-parabolic end} 
The following example uses vertical lines instead of horizontal lines as in the horospheres

\begin{example}[cylinders]
\begin{figure}[H] 
   \centering
   \includegraphics[scale=0.25]{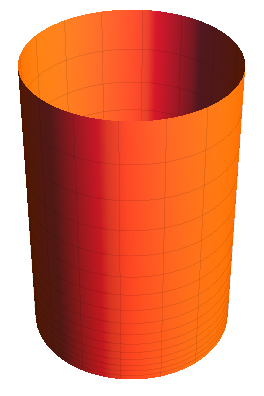}\quad 
   \includegraphics[scale=0.25]{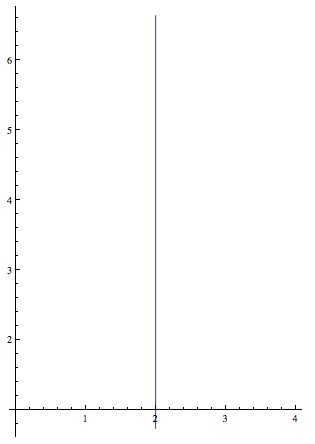}
   \caption{The cylinder immersed in $\HT$ with parameters $b=2, c=1$. This surface has two ends, one parabolic and the other non-parabolic.}
   \label{fig:example4}
\end{figure}
 The family of parameterized curves in $\HT$
$$
\beta_{b,c}:(-\infty,\infty)\to\HT,\quad\beta_{b,c}(t) = (b,0, c \cdot e^t)
$$
can be rotated over the vertical axes $(0,x_3)$ to get a cylinder (see figure \ref{fig:example4}). This surface of revolution in $\HT$ has two ends of revolution in $\HT$. One upper end $E^+$ given by the rotation of the parametrized by arc length curve 
$$
\gamma_{E^+}:[0,\infty)\to\HT,\quad\gamma_{b,c}(s) = (b,0, c \cdot e^s)
$$
And an other end $E^-$ obtained by the rotation of the parametrized by arc length curve 
$$
\gamma_{E^-}:[0,\infty)\to\HT,\quad\gamma_{b,c}(s) = (b,0, c \cdot e^{-s})
$$
Observe that the end $E^+$ is on the $\frac{c}{b}$-cone, and hence by Theorem \ref{teoendh3} is a parabolic end. On the other hand the end $E_{-}$ has
$$
w_{-1}(s) = \frac{b}{c}e^s
$$
thus
$$
\int^{\infty}_0 \frac{1}{w_{-1}(s)} ds = \int^{\infty}_0 \frac{c}{b}e^{-s} ds = \frac{c}{b}<\infty
$$
the end $E_{-}$ is therefore non-parabolic by using proposition \ref{prop2.16}.
\end{example}
\vspace{0.5cm}

\def\cprime{$'$} \def\cprime{$'$} \def\cprime{$'$} \def\cprime{$'$}
  \def\cprime{$'$} \def\cprime{$'$} \def\cprime{$'$} \def\cprime{$'$}
  \def\cprime{$'$} \def\cprime{$'$}
\providecommand{\bysame}{\leavevmode\hbox to3em{\hrulefill}\thinspace}
\providecommand{\MR}{\relax\ifhmode\unskip\space\fi MR }
\providecommand{\MRhref}[2]{%
  \href{http://www.ams.org/mathscinet-getitem?mr=#1}{#2}
}
\providecommand{\href}[2]{#2}

\end{document}